\documentclass{amsart}
\usepackage{amssymb}
\usepackage{float}
\usepackage{verbatim}

\usepackage{array}
\usepackage[margin=1in]{geometry}
\makeatletter
\newcommand{\thickhline}{%
	\noalign {\ifnum 0=`}\fi \hrule height 1pt
	\futurelet \reserved@a \@xhline
}
\newcolumntype{"}{@{\hskip\tabcolsep\vrule width 1pt\hskip\tabcolsep}}
\makeatother

\newtheorem{theorem}{Theorem}[section]
\newtheorem{lemma}[theorem]{Lemma}

\theoremstyle{definition}
\newtheorem{definition}[theorem]{Definition}

\newtheorem{corollary}[theorem]{Corollary}

\newtheorem{example}[theorem]{Example}

\theoremstyle{remark}
\newtheorem{remark}[theorem]{Remark}

\numberwithin{equation}{section}

\begin{document}
	\title{On the Hilbert Series of the tangent cones for some 4-generated pseudo symmetric monomial curves}
	\author{N\.{i}l \c{S}ah\.{i}n}
	\address{Department of Industrial Engineering, Bilkent University, Ankara, 06800 Turkey}
	\email{nilsahin@bilkent.edu.tr}

	\thanks{}
	
	
	\date{\today}
	
	\commby{}
	
	\dedicatory{}

\begin{abstract} In this article, we study Hilbert Series of  non-Cohen-Maculay tangent cones for some 4-generated pseudo symmetric monomial curves. We show that the Hilbert Function is nondecreasing by explicitly computing it. We also compute standard bases of these toric ideals.
	
\end{abstract}

\maketitle
\section{Introduction}
\label{Sec:1}
Hilbert function of the tangent cone of a projective variety carries  a lot of information about the discrete invariants of the variety and of its embedding such as  dimension, the degree and the arithmetic genus of the variety. Though the Hilbert function of a Cohen-Macaulay graded ring is well known with the help of Macaulay's theorem, the same thing do not hold for the Hilbert Function of its tangent cone, \cite{Elias}.  Some partial results obtained in  special cases about the growth of the Hilbert Function of the tangent cone see \cite{am,AMS,AOS,mz,Oneto,Puthen}.

In this paper we will work with one dimensional Cohen-Macaulay local rings. It is known that when the tangent cone is also Cohen-Macaulay,  the Hilbert Function of the local ring is non-decreasing. Rossi's conjecture states that  ``The Hilbert function of a Gorenstein Local ring of dimension one is non-decreasing''. This conjecture is still open in embedding dimension 4 even for monomial curves. Nondecreasingness of the Hilbert Function of the tangent cone for some 4 generated symmetric monomial curves is studied by Arslan and Mete in \cite{am}  and Katsabekis in \cite{an}.  Arslan and Mete put an extra condition on the genarators, namely $\alpha_2\leq \alpha_{21}+\alpha_{24}$ and showed that the tangent cone is Cohen Macaulay and, as a result, they showed that Hilbert Function is non-decreasing without the need of explicit Hilbert Function computation. Katsabekis also studied Cohen-Macaulay Tangent cone case and computed Hilbert Function explicitly. For 4 generated symmetric monomial curves, the case   $\alpha_2> \alpha_{21}+\alpha_{24}$ ( ie. not Cohen-Macaulay Tangent cone) is still open. We studied 4-generated pseudo-symmetric monomial curves with Cohen Macaulay tangent cones ($\alpha_2\leq \alpha_{21}+1$) in \cite{SahinSahin} and showed non-decreasingness of the Hilbert Function.  Without Cohen-Macaulayness of the tangent cone ($\alpha_2>\alpha_{21}+1$), non-decreasingness of the Hilbert Function of the local ring is not guaranteed and hence requires an explicit Hilbert Function computation. With an observation that the number of generators in the standard basis increase when $\alpha_4$ increse, we studied the simplest case $\alpha_4=2$ in \cite{Sahin} . We  showed that the number of elements in the standard basis depends on a parameter k and the Hilbert function is non-decreasing when this parameter $k$ is $1$. 
In this paper, we focus on the next case, $\alpha_4=3$.  As the standard basis computation requires adding the normal forms of s-polynomials of all the elements in the ideal, it is important to see the polynomials to start with. Understanding this next case a step towards understanding the general case using induction.
This case differs from the case $\alpha_4=2$ as can be seen in the next two examples. Though the parameter $k=2$ in both of these examples, the number of elements and their forms are much different. The following examples are done using SINGULAR. \footnote{Singular 2.0. A Computer Algebra System for Polynomial Computations. Available at http://www.singular.uni-kl.de.}
\begin{example}
	For $\alpha_{21}=4$, $\alpha_{1}=22$,$\alpha_{2}=13$, $\alpha_{3}=5$, $\alpha_{4}=2$, we have $k=2$ and the corresponding standard basis is $\{ X_1^{22}-X_3X_4
	,X_2^{13}-X_1^4X_4
	,X_3^5-X_1^{17}X_2
	,X_4^2-X_1X_2^{12}X_3^4
	,X_1^5X_3^4-X_2X_4
	,X_1^{26}-X_2^{13}X_3
	,X_1^9X_3^4-X_2^{14}
	,X_1^{35}X_3^3-X_2^{27}  \}$.

\end{example}
\begin{example}
		For $\alpha_{21}=101$, $\alpha_{1}=501$,$\alpha_{2}=340$, $\alpha_{3}=18$, $\alpha_{4}=3$, we have $k=2$ and the corresponding standard basis is
$\{X_2X_4^2-X_1^{102}X_3^{17}, X_3X_4^2-X_1^{501},                  X_4^3-X_1X_2^{339}X_3^{17}  ,X_3^{18}-X_1^{399}X_2 ,X_1^{101}X_4-X_2^{340},X_1^{203}X_3^{17}-X_2^{341}X_4  ,X_2^{340}X_3X_4-X_1^{602} , X_2^{680}X_3-X_1^{703},  X_1^{906}X_3^{16}-X_2^{1021}X_4, X_1^{1609}X_3^{15}-X_2^{1701}X_4, \newline X_1^{2312}X_3^{14}-X_2^{2381}X_4, X_1^{3015}X_3^{13}-X_2^{3061}X_4  ,X_1^{3718}X_3^{12}-X_2^{3741}X_4 ,X_2^{4421}X_4-X_1^{4421}X_3^{11} ,  X_1^{4522}X_3^{11}-X_2^{4761}, \newline X_1^{5225}X_3^{10}-X_2^{5441},  X_1^{5928}X_3^9-X_2^{6121},    X_1^{6631}X_3^8-X_2^{6801}, X_1^{7334}X_3^7-X_2^{7481},   X_1^{8037}X_3^6-X_2^{8161}, X_1^{8740}X_3^5-X_2^{8841}, X_1^{9443}X_3^4-X_2^{9521} , X_1^{10146}X_3^3-X_2^{10201}, X_1^{10849}X_3^2-X_2^{10881}    , X_1^{11552}X_3-X_2^{11561} , X_2^{12241}-X_1^{12255}  \}$

\end{example}
This shows that the standard basis do not only depend on $k$ like in $\alpha_4=2$ case.
\section{Basic Definitions}
Using the same notation with \cite{Sahin},  Let $S$ be the numerical semigroup $S=\langle n_1,\dots,n_k \rangle=\{ \displaystyle\sum_{i=1}^{k} u_in_i | u_i \in \mathbb{N}\}$, where $n_1< n_2<\dots<n_k$  are positive integers with $\gcd (n_1,\dots,n_k)=1$ . Let $K[S]=K[t^{n_1}, t^{n_2}, \dots, t^{n_k}]$ be the semigroup ring of $S$, where $K$ is algebraically closed field,  and $A=K[X_1,X_2,\dots,X_k]$. If  $\phi: A {\longrightarrow} K[S]$ with $\phi(X_i)=t^{n_i}$ and $\ker \phi=I_S$ , then $K[S]\simeq A/I_S$. $C_S$ is the affine curve with parametrization 
$$X_1=t^{n_1},\ \ X_2=t^{n_2},\ \dots,\  X_k=t^{n_k}  $$
corresponding to $S$ , and $I_S$ is the defining ideal of $C_S$.  $n_1$ is  the \textit{multiplicity} of $C_S$. $R_S=K[[t^{n_1},\dots,t^{n_k}]]$  is the local ring with the maximal ideal  $\mathfrak{m}=\langle t^{n_1},\hdots,t^{n_k}\rangle$. Then
$gr_{\mathfrak{m}}(R_S)=\bigoplus_{i=0}^{\infty} \mathfrak{m}^i/\mathfrak{m}^{i+1}\cong A/{I^*_S},$ is the associated graded ring 
where ${I^*_S}=\langle f^*| f \in I_S \rangle$ with $f^*$ denoting the least homogeneous summand of $f$. 

We mean the Hilbert function of the associated
graded ring $gr_{\mathfrak{m}}(R_S)=\bigoplus_{i=0}^{\infty} \mathfrak{m}^i/\mathfrak{m}^{i+1}$ by the  Hilbert function $H_{R_S}(n)$ of the local ring $R_S$ .
That is,
$$H_{R_S}(n)=H_{gr_{\mathfrak{m}}(R_S)}(n)=dim_{R_S/\mathfrak{m}}(\mathfrak{m}^n/\mathfrak{m}^{n+1}) \; \; n\geq
0.$$

The Hilbert series of $R_S$ is defined to be the generating function
$$HS_{R_S}(t)=\begin{displaystyle}\sum_{n \in \mathbb{N}}\end{displaystyle}H_{R_S}(n)t^n.$$
By the Hilbert-Serre theorem it can also be written as:
$HS_{R_S}(t)=\frac{P(t)}{(1-t)^k}=\frac{Q(t)}{(1-t)^d}$, where $P(t)$ and $Q(t)$ are polynomials with coefficients in
	$\mathbb{Z}$ and $d$ is the Krull dimension of $R_S$. $P(t)$ is called first Hilbert Series and $Q(t)$ is called second Hilbert series, \cite{greuel-pfister,Rossi}. It is also known that there is
a polynomial $P_{R_S}(n) \in \mathbb{Q}[n]$ called the Hilbert polynomial of $R_S$ such that
$H_{R_S}(n)=P_{R_S}(n)$ for all $n \geq n_0$, for some $n_0 \in \mathbb{N}$. The smallest $n_0$ satisfying this
condition is the regularity index of the Hilbert function of $R_S$.

A $4$-generated semigroup $S=\langle n_1,n_2,n_3,n_4 \rangle$ is pseudo-symmetric if and only if there are integers $\alpha_i>1$, for
$1\le i\le4$, and $\alpha_{21}>0$ with $0<\alpha_{21}<\alpha_1-1$,
such that 
\begin{eqnarray*}
	n_1&=&\alpha_2\alpha_3(\alpha_4-1)+1,\\
	n_2&=&\alpha_{21}\alpha_3\alpha_4+(\alpha_1-\alpha_{21}-1)(\alpha_3-1)+\alpha_3,\\
	n_3&=&\alpha_1\alpha_4+(\alpha_1-\alpha_{21}-1)(\alpha_2-1)(\alpha_4-1)-\alpha_4+1,\\
	n_4&=&\alpha_1\alpha_2(\alpha_3-1)+\alpha_{21}(\alpha_2-1)+\alpha_2. 
\end{eqnarray*}
Then, the toric ideal is $I_S=\langle f_1,f_2,f_3,f_4,f_5 \rangle$ with
\begin{eqnarray*} f_1&=&X_1^{\alpha_1}-X_3X_4^{\alpha_4-1},  \quad \quad
	f_2=X_2^{\alpha_2}-X_1^{\alpha_{21}}X_4, \quad
	f_3=X_3^{\alpha_3}-X_1^{\alpha_1-\alpha_{21}-1}X_2,\\
	f_4&=&X_4^{\alpha_4}-X_1X_2^{\alpha_2-1}X_3^{\alpha_3-1}, \quad 
	f_5=X_1^{\alpha_{21}+1}X_3^{\alpha_3-1}-X_2X_4^{\alpha_4-1}.
\end{eqnarray*} 
See \cite{komeda} for the details.

We focus on the case $\alpha_4=3.$
\section{Standard bases}\label{2}
\begin{remark}\label{f7}
	If $n_1<n_2$ then $2\alpha_2+1<2\alpha_{21}+\alpha_1$
\end{remark}
\begin{proof}
	
	$$	\begin{array}{llll}
	&n_1&<&n_2\\
	\implies	&2\alpha_2\alpha_3+1&<&3\alpha_{21}\alpha_3+(\alpha_1-\alpha_{21}-1)(\alpha_3-1)+\alpha_3\\
	\implies&2\alpha_2\alpha_3+1&<&2\alpha_{21}\alpha_3+\alpha_1\alpha_3-\alpha_1+\alpha_{21}+1\\
	\implies&\alpha_3(2\alpha_2-2\alpha_{21}-\alpha_1)&<&\alpha_{21}-\alpha_1 \\
	\implies&\alpha_3(2\alpha_2-2\alpha_{21}-\alpha_1+1)&<&\alpha_3+\alpha_{21}-\alpha_1<0 \ \   {\rm by} (2) \\
	\implies&2\alpha_2-2\alpha_{21}-\alpha_1+1&<&0 \\
	\implies&2\alpha_2+1&<&2\alpha_{21}+\alpha_1 \\
	\end{array}$$
	
\end{proof}
If $n_1<n_2<n_3< n_4$ then it is known from \cite{SahinSahin} that 
\begin{enumerate}
	\item[(1)] $\alpha_1>\alpha_4$
	\item[(2)] $\alpha_3<\alpha_1-\alpha_{21}$
	\item[(3)] $\alpha_4<\alpha_2+\alpha_3-1$
\end{enumerate}
and these conditions completely determine the leading monomials of $f_1, f_3$ and $f_4$. Indeed, $ {\rm LM}( f_1)= X_3X_4^{\alpha_{4}-1}$ by $(1)$, ${\rm LM}(f_3)= X_3^{\alpha_3}$ by $(2)$, ${\rm LM}(f_4)=X_4^{\alpha_4}$ by $(3)$
If we also let 
\begin{enumerate}
	\item[(4)]  $\alpha_2>\alpha_{21}+1$
\end{enumerate}
then ${\rm LM}(f_2)=X_1^{\alpha_{21}}X_4$ by $(4)$. Then if $n_1<n_2<n_3<n_4$, using the remarks in \cite{Sahin}, 

\begin{enumerate}
	\item[(5)] $\alpha_{21}+\alpha_3 > \alpha_4$ 
\end{enumerate}

\begin{enumerate}
	\item[(6)] $\alpha_1+\alpha_{21}+1\geq \alpha_2+\alpha_4$
\end{enumerate} 
Now using remark \ref{f7},

\begin{enumerate}
	\item[(7)] $2\alpha_2+1 < 2\alpha_{21}+\alpha_1$
\end{enumerate} 
These determine ${\rm LM}(f_5)$, and will determine ${\rm LM}(f_6)$ and ${\rm LM}(f_7)$. We know that the standard basis when $\alpha_4=2$ depends on the parameter $k$. We will show that the standard basis when $\alpha_4=3$ depends on three parameters $k$, $s$ and $l$ defined as follows:
\begin{definition}
	Define $k,l$ and $s$ as the smallest integers satisfying 
	$$(k-1)\alpha_1+(k+1)\alpha_{21}+\alpha_3>k\alpha_2+(k+1)$$
	
	$$s\alpha_1+(2s+2)\alpha_{21}+\alpha_3 > (2s+1)\alpha_2+(s+2)$$
	
	$$l\alpha_1+(2l+3)\alpha_{21}+\alpha_3 > (2l+2)\alpha_2+(l+1)$$ 
	
respectively.
\end{definition}
\begin{remark}
	$s$, $k$ and $l$ always exist.
\end{remark}
\begin{proof}
Assume to the contrary that $s$ does not exist. Then for any integer $i$ , $i\alpha_1+(2i+2)\alpha_{21}+\alpha_3\leq(2i+1)\alpha_2+i+2\implies i(\alpha_1+2\alpha_{21}-\alpha_2-1)\leq \alpha_2+2-\alpha_3-2\alpha_{21}$   which gives a contradiction since the right hand side of this inequality is fixed and $\alpha_1+2\alpha_{21}-\alpha_2-1$ is positive by $(6)$.

Using a similar argument,  $k$ and $l$ always exist.
\end{proof}

\begin{remark}
	If $n_1<n_2<n_3<n_4$, then $k$ is at most $2$.
\end{remark}
\begin{proof}
	Assume to the contrary that $k>2$. Then 
$$\begin{array}{rllll}
	2(\alpha_2+1)&\geq & \alpha_1+3\alpha_{21}+\alpha_3-1&&\\
	2\alpha_2&\geq&\alpha_1+3\alpha_{21}+\alpha_3-3&>\alpha_1+\alpha_{21}&
\end{array}$$
Then, \begin{equation}\label{eq1}
\alpha_1+\alpha_{21}-2\alpha_2<0
\end{equation}
On the other hand, $n_1<n_2$ implies

$$\begin{array}{rllll}
2\alpha_2\alpha_3+1&<&2\alpha_{21}\alpha_3+(\alpha_1-\alpha_{21}-1)(\alpha_3-1)+\alpha_3&&\\
\alpha_1-\alpha_{21}&<&\alpha_3(\alpha_1+\alpha_{21}-2\alpha_2)&<&0\\
\alpha_1-\alpha_{21}&<&0&&
\end{array}$$

which is a contradiction. Hence, $k$ can not exceed $2$.
	\end{proof}

\begin{remark}
	$s\leq l$
\end{remark}
\begin{proof}
We know by $(4)$ that $\alpha_2-\alpha_{21}-1> 0$. If $l$ satisfies $l\alpha_1+(2l+3)\alpha_{21}+\alpha_3 > (2l+2)\alpha_2+(l+1)$ then $l\alpha_1+(2l+2)\alpha_{21}+\alpha_3 > (2l+1)\alpha_2+(l+2)+\alpha_2-\alpha_{21}-1>(2l+1)\alpha_2+(l+2)$. Since s is the smallest integer satisfying this inequality, we must have $s\leq l$ 
	\end{proof}
	\begin{remark}
	If $k=1$, then $s=0$
	\end{remark}
\begin{proof}
If $k=1$, then $\alpha_2+1< 2\alpha_{21}+\alpha_3-1$ which is $s\alpha_1+(2s+2)\alpha_{21}+\alpha_3 > (2s+1)\alpha_2+s+2.$ for $s=0$. Hence if $k=1$, then $s=0$.
\end{proof}

\begin{theorem}\label{mainthm}
	Let $S=\langle n_1,n_2,n_3,n_4 \rangle$ be a 4-generated pseudosymmetric numerical semigroup with $n_1<n_2<n_3<n_4$ and $\alpha_2>\alpha_{21}+1$. If $k, s $ and $l$ be  defined as above,then the standard basis for $I_S$ is
	$$G=\{f_1,f_2,f_3,f_4,f_5,f_6,f_7,g_0,g_1,g_2,...,g_s,h_{s},h_{s+1}...,h_{l}\}$$
	where $f_6=X_1^{\alpha_1+\alpha_{21}}-X_2^{\alpha_2}X_3X_4$, $f_{7}=X_1^{\alpha_1+2\alpha_{21}}-X_2^{2\alpha_2}X_3$ , $g_i=X_2^{(2i+1)\alpha_{2}+1}X_4-X_1^{i\alpha_1+(2i+2)\alpha_{21}+1}X_3^{\alpha_3-(i+1)}$ and $h_j=X_2^{(2j+2)\alpha_{2}+1}-X_1^{j\alpha_1+(2j+3)\alpha_{21}+1}X_3^{\alpha_3-(j+1)}$
\end{theorem}
Before we prove the theorem let's state and prove the next lemma:
\begin{lemma}{\label{gj}}
${\rm NF}(g_j|G)=0$ for $j>s$.
\end{lemma}
\begin{proof}
$T_{g_j}=\{g_s\}$ and ${\rm spoly}(g_j,g_s)=X_1^{s\alpha_1+(2s+2)\alpha_{21}+1}X_2^{2(j-s)\alpha_2}X_3^{\alpha_3-(s+1)}-X_1^{j\alpha_1+(2j+2)\alpha_{21}+1}X_3^{\alpha_3-(j+1)}=r_1$. $T_{r_1}=\{f_7\}$ and ${\rm spoly}(r_1,f_7)=X_1^{(s+1)\alpha_1+(2s+4)\alpha_{21}+1}X_2^{2(j-s-1)\alpha_2}X_3^{\alpha_3-(s+2)}-X_1^{j\alpha_1+(2j+2)\alpha_{21}+1}X_3^{\alpha_3-(j+1)}=r_2.$ $T_{r_2}=\{f_7\}$ and continuing inductively $r_{j-s}={\rm spoly}(r_{j-s-1},f_7)=X_1^{(j-1)\alpha_1+2j\alpha_{21}+1}X_3^{\alpha_3-(j+1)}f_7$. Hence, ${\rm NF}(g_j|G)=0$.
\end{proof}
Now we are ready for the proof of the theorem \ref{mainthm}.
\begin{proof}
	We will prove the theorem by applying standard basis algorithm with NFM\tiny{ORA} \normalsize  as the normal form algorithm, see \cite{greuel-pfister}. Here $G=\{f_1,f_2,f_3,f_4,f_5,f_6,f_7,g_0,g_1,g_2,...,g_s,h_{s},h_{s+1}...,h_{t}\}$ and $T_h$ denotes the set $\{g \in G : {\rm LM}(g) \mid {\rm LM}(h)\}$ and ${\rm ecart}(h)$ is
	${deg}(h)-{deg}({\rm LM}(h))$. Note that ${\rm LM}(f_6)=X_2^{\alpha_2}X_3X_4$ by $(6)$ , ${\rm LM}(g_i)=X_1^{i\alpha_1+(2i+2)\alpha_{21}+1}X_3^{\alpha_3-(i+1)}$ for all $i<s$, ${\rm LM}(g_s)=X_2^{(2s+1)\alpha_{2}+1}X_4$, ${\rm LM}(h_j)=X_1^{j\alpha_1+(2j+3)\alpha_{21}+1}X_3^{\alpha_3-(j+1)}$ for all $s\leq j < l$ and ${\rm LM}(h_l)=X_2^{(2l+2)\alpha_{2}+1}$ by the definitions of $s$ and $l$ .

	\begin{center}
		\underline{ For $k=1$:}
	\end{center}
	In this case $g_0=X_1^{2\alpha_{21}+1}X_3^{\alpha_3-1}-X_2^{\alpha_2+1}X_4$ and $\alpha_2+1<2\alpha_{21}+\alpha_3$ which implies that ${\rm LM}(g_0)=X_2^{\alpha_2+1}X_4$ . We need to show that $NF(\text{spoly}(f_m,f_n) \vert G)=0$ for all $m,n$ with $1 \leq m < n \leq 6$ .
	
	\begin{itemize}
		\item  ${\rm spoly}(f_1,f_2)=f_6$ and hence  $NF({\rm spoly}(f_1,f_2) \vert G)=0$
		\item   ${\rm spoly}(f_1,f_3)=X_1^{\alpha_{1}}X_3^{\alpha_3-1}-X_1^{\alpha_1-\alpha_{21}-1}X_2X_4^2$ and ${\rm LM}({\rm spoly}(f_1,f_{3}))=X_1^{\alpha_1-\alpha_{21}-1}X_2X_4^2$ by $(5)$. Let $r_1={\rm spoly}(f_1,f_3)$. If $\alpha_{1}<2\alpha_{21}+1$ then $T_{r_1}=\{f_5\}$ and since ${\rm spoly}(r_1,g)=0$,  $NF({\rm spoly}(f_1,f_3) \vert G)=0$. Otherwise $T_{r_1}=\{f_2\}$ and ${\rm spoly}(r_1,f_2)=X_1^{\alpha_1-2\alpha_{21}-1}X_2^{\alpha_2+1}X_4-X_1^{\alpha_{1}}X_3^{\alpha_{3}-1}$. Set $r_2={\rm spoly}(r_1,f_2)$, ${\rm LM}(r_2)=X_1^{\alpha_1-2\alpha_{21}-1}X_2^{\alpha_2+1}X_4$ and $T_{r_2}=\{f_7\}$ and  ${\rm spoly}(r_2,f_7)=0$, hence $NF({\rm spoly}(f_1,f_3) \vert G)=0$ . 
		\item   ${\rm spoly}(f_1,f_4)=X_1^{\alpha_1}X_4-X_1X_2^{\alpha_{2}-1}X_3^{\alpha_3}$. Set $r_1={\rm spoly}(f_1,f_4)$. If ${\rm LM}(r_1)=X_1^{\alpha_1}X_4$ then $T_{r_1}=\{f_2\}$ and ${\rm spoly}(r_1,f_2)=X_1X_2^{\alpha_2-1}f_3$. If ${\rm LM}(r_1)=X_1X_2^{\alpha_{2}-1}X_3^{\alpha_3}$ then $T_{r_1}=\{f_3\}$ and ${\rm spoly}(r_1,f_3)=X_1^{\alpha_1-\alpha_{21}}f_2$. Hence in both cases, $NF({\rm spoly}(f_1,f_4) \vert G)=0$
		\item   ${\rm spoly}(f_1,f_5)=X_1^{\alpha_{21}+1}X_3^{\alpha_3}-X_1^{\alpha_{1}}X_2=X_1^{\alpha_{21}+1}f_3$ hence $NF({\rm spoly}(f_1,f_5) \vert G)=0$
		\item   ${\rm spoly}(f_1,f_6)=X_1^{\alpha_1}f_2$ and hence $NF({\rm spoly}(f_1,f_6) \vert G)=0$
		\item  ${\rm spoly}(f_1,f_7)=X_1^{\alpha_1+2\alpha_{21}}X_4^2-X_1^{\alpha_1}X_2^{2\alpha_2}=r_1$ then ${\rm LM}(r_1)=X_1^{\alpha_1+2\alpha_{21}}X_4^2$ by $(4)$ and $T_{r_1}=\{f_2\}$. Then ${\rm spoly}(r_1,f_2)=X_1^{\alpha_1}X_2^{\alpha_{2}}f_2$. Hence $NF({\rm spoly}(f_1,f_7) \vert G)=0$.
		\item $NF({\rm spoly}(f_2,f_3) \vert G)=0$ as ${\rm LM}(f_2)$ and ${\rm LM}(f_3)$ are relatively prime.
		\item   ${\rm spoly}(f_2,f_4)=X_2^{\alpha_2-1}f_5$ and hence $NF({\rm spoly}(f_2,f_4) \vert G)=0$
		\item   ${\rm spoly}(f_2,f_5)=g_0$ and hence $NF({\rm spoly}(f_2,f_5) \vert G)=0$
		\item  ${\rm spoly}(f_2,f_6)=f_7$    and hence $NF({\rm spoly}(f_2,f_6) \vert G)=0$.
		\item  $NF({\rm spoly}(f_2,f_7) \vert G)=0$ as ${\rm LM}(f_2)$ and ${\rm LM}(f_7)$ are relatively prime. 
	
		\item $NF({\rm spoly}(f_3,f_4) \vert G)=0$ as ${\rm LM}(f_3)$ and ${\rm LM}(f_4)$ are relatively prime.
		\item  $NF({\rm spoly}(f_3,f_5) \vert G)=0$ as ${\rm LM}(f_3)$ and ${\rm LM}(f_4)$ are relatively prime.
	
		\item   ${\rm spoly}(f_3,f_6)=X_1^{\alpha_1-\alpha_{21}-1}g_0$ and hence $NF({\rm spoly}(f_3,f_6) \vert G)=0$

		\item   ${\rm spoly}(f_3,f_7)=X_1^{\alpha_1-\alpha_{21}-1}h_0$ (since $k=1$, $s=0$ and $h_0 \in G$) .  Hence, $NF({\rm spoly}(f_3,f_7) \vert G)=0$
		\item   ${\rm spoly}(f_4,f_5)=X_1X_3^{\alpha_3-1}f_2$ and hence $NF({\rm spoly}(f_4,f_5) \vert G)=0$
	
		\item  ${\rm spoly}(f_4,f_6)=X_1X_2^{2\alpha_2-1}X_3^{\alpha_3}-X_1^{\alpha_1+\alpha_{21}}X_4^2=r_1$. If ${\rm LM}(r_1)=X_1X_2^{2\alpha_2-1}X_3^{\alpha_3}$; $T_{r_1}=\{f_3\}$ and ${\rm spoly}(f_3,r_1)=X_1^{\alpha_1+\alpha_{21}}X_4^2-X_1^{\alpha_1-\alpha_{21}}X_2^{2\alpha_2}=r_2$. ${\rm LM}(r_2)=X_1^{\alpha_1+\alpha_{21}}X_4^2$ and $T_{r_2}=\{f_2\}$. Then ${\rm spoly}(f_2,r_2)=0$. If ${\rm LM}(r_1)=X_1^{\alpha_1+\alpha_{21}}X_4^2$ ;$T_{r_1}=\{f_2\}$ and ${\rm spoly}(f_2,r_1)=X_1^{\alpha_1}X_2^{\alpha_2}X_4-X_1X_2^{2\alpha_2-1}X_3^{\alpha_3}=r_2$. ${\rm LM}(r_2)=X_1^{\alpha_1}X_2^{\alpha_2}X_4$ and $T_{r_2}=\{f_3\}$. Then ${\rm spoly}(f_3,r_2)=0$. 	Hence, in both cases $NF({\rm spoly}(f_4,f_6) \vert G)=0$
	
		\item	$NF({\rm spoly}(f_4,f_7) \vert G)=0$ as ${\rm LM}(f_4)$ and ${\rm LM}(f_7)$ are relatively prime.
	
		\item   ${\rm spoly}(f_5,f_6)=X_1^{\alpha_{21}+1}X_2^{\alpha_2-1}X_3^{\alpha_3}-X_1^{\alpha_{1}+\alpha_{21}}X_4$ and let $r_1= {\rm spoly}(f_5,f_6)$. If ${\rm LM}(r_1)=X_1^{\alpha_{21}+1}X_2^{\alpha_2-1}X_3^{\alpha_3}$ then $T_{r_1}=\{f_3\}$ and ${\rm spoly}(r_1,f_3)=X_1^{\alpha_1-\alpha_{21}-1}f_2$ and hence $NF({\rm spoly}(f_5,f_6) \vert G)=0$. If ${\rm LM}(r_1)= X_1^{\alpha_{1}+\alpha_{21}}X_4$ then $T_{r_1}=\{f_2\}$ and ${\rm spoly}(r_1,f_2)=X_2^{\alpha_2-1}f_3$ and hence $NF({\rm spoly}(f_5,f_6) \vert G)=0$

		\item  ${\rm spoly}(f_5,f_7)=X_1^{\alpha_1+2\alpha_{21}}X_4^2-X_1^{\alpha_{21}+1}X_2^{2\alpha_2-1}X_3^{\alpha_3}=r_1$.  \newline
		If ${\rm LM}(r_1)=X_1^{\alpha_1+2\alpha_{21}}X_4^2$, then $T_{r_1}=f_2$ and ${\rm spoly}(r_1,f_2)= X_1^{\alpha_1+\alpha_{21}}X_2^{\alpha_2}X_4-X_1^{\alpha_{21}+1}X_2^{2\alpha_2-1}X_3^{\alpha_3}=r_2$. Depending on the leading monomial of $r_2$, $T_{r_2}$ is either $f_2$ or $f_3$. If it is $f_2$,  ${\rm spoly}(r_2,f_2)=X_1^{\alpha_{21}+1}X_2^{2\alpha_2-1}f_3$. If it is $f_3$, ${\rm spoly}(r_2,f_3)=X_1^{\alpha_{1}}X_2^{\alpha_2}f_2$ and hence $NF({\rm spoly}(f_5,f_7) \vert G)=0$.\newline
		If ${\rm LM}(r_1)=X_1^{\alpha_{21}+1}X_2^{2\alpha_2-1}X_3^{\alpha_3}$, then $T_{r_1}=f_3$ and ${\rm spoly}(r_1,f_3)=X_1^{\alpha_1}X_2^{2\alpha_2}-X_1^{\alpha_1+2\alpha_{21}}X_4^2=r_2.$  ${\rm LM}(r_2)= X_1^{\alpha_1+2\alpha_{21}}X_4^2$ and $T_{r_2}=\{f_2\}$. Then, ${\rm spoly}(r_2,f_2)= X_1^{\alpha_1}X_2^{\alpha_2}f_2$. Hence, $NF({\rm spoly}(f_5,f_7) \vert G)=0$

		\item  ${\rm spoly}(f_6,f_7)=X_1^{\alpha_1+\alpha_{21}}f_2$ and hence,  $NF({\rm spoly}(f_6,f_7) \vert G)=0$

		\item ${\rm spoly}(f_1,g_0)=X_1^{\alpha_1}X_2^{\alpha_2+1}-X_1^{2\alpha_{21}+1}X_3^{\alpha_3}X_4=r_1$. Using $(2)$ and $(4)$, ${\rm LM}(r_1)=X_1^{2\alpha_{21}+1}X_3^{\alpha_3}X_4$ and $T_{r_1}=\{f_2,f_3\}$. If ${\rm ecart}(f_2)$ is minimal, ${\rm spoly}(r_1,f_2)= X_1^{\alpha_{21}+1}X_2^{\alpha_2}f_3$ . If ${\rm ecart}(f_3)$ is minimal, ${\rm spoly}(r_1,f_3)=X_1^{\alpha_1}X_2f_2$. Hence, in both cases $NF({\rm spoly}(f_1,g_0) \vert G)=0$
		
		\item$NF({\rm spoly}(f_2,g_0) \vert G)=0$ as ${\rm spoly}(f_2,g_0)=h_0$  and $h_0 \in G$.
			
		\item   $NF({\rm spoly}(f_3,g_0) \vert G)=0$ as   ${\rm LM}(f_3)$ and ${\rm LM}(g_0)$ are relatively prime.
				
		\item ${\rm spoly}(f_4,g_0)=X_1^{2\alpha_{21}+1}X_3^{\alpha_3-1}X_4^2-X_1X_2^{2\alpha_2}X_3^{\alpha_3-1}=r_1$. ${\rm LM}(r_1)=X_1^{2\alpha_{21}+1}X_3^{\alpha_3-1}X_4^2$ and $T_{r_1}=\{f_1,f_2\}$ but ${\rm ecart}(f_2)$ is minimal. ${\rm spoly}(r_1,f_2)=X_1X_2^{\alpha_2}X_3^{\alpha_3-1}f_2$ and hence $NF({\rm spoly}(f_4,g_0) \vert G)=0$ 
					
		\item ${\rm spoly}(f_5,g_0)=X_1^{\alpha_{21}+1}X_3^{\alpha_3-1}f_2$ and hence $NF({\rm spoly}(f_5,g_0) \vert G)=0$ 
						
		\item ${\rm spoly}(f_6,g_0)=X_1^{2\alpha_{21}+1}f_3$ and hence $NF({\rm spoly}(f_6,g_0) \vert G)=0$ 
							
	    \item ${\rm spoly}(f_7,g_0)= X_1^{2\alpha_{21}+1}X_2^{\alpha_2-1}X_3^{\alpha_3}-X_1^{\alpha_1+2\alpha_{21}}X_4=r_1$. If ${\rm LM}(r_1)=X_1^{2\alpha_{21}+1}X_2^{\alpha_2-1}X_3^{\alpha_3}$, then $T_{r_1}=\{f_3\}$ and ${\rm spoly}(r_1,f_3)=X_1^{\alpha_1+\alpha_{21}}f_2$ \\
	    If ${\rm LM}(r_1)=X_1^{\alpha_1+2\alpha_{21}}X_4$, then $T_{r_1}=\{f_2\}$ and ${\rm spoly}(r_1,f_2)=X_1^{2\alpha_{21}+1}X_2^{\alpha_2-1}f_3$ so in both cases, $NF({\rm spoly}(f_7,g_0) \vert G)=0$.
	    							
	    \item $NF({\rm spoly}(g_0,h_j)\vert G)=0$ as ${\rm LM}(g_0)$ and ${\rm LM}(h_j)$ are relatively prime for all $0\leq j <l$.

	   \item ${\rm spoly}(g_0,h_l)=X_1^{t\alpha_1+(2l+3)\alpha_{21}+1}X_3^{\alpha_3-(l+1)}X_4-X_1^{2\alpha_{21}+1}X_2^{(2l+1)\alpha_2}X_3^{\alpha_3-1}=r_1$ and ${\rm LM}(r_1)=$\newline$ X_1^{2\alpha_{21}+1}X_2^{(2l+1)\alpha_2}X_3^{\alpha_3-1}$. $T_{r_1}=\{f_7\}$ and ${\rm spoly}(r_1,f_7)=X_1^{\alpha_1+
	   	4\alpha_{21}+1}X_3^{\alpha_3-(l+1)}\left[ X_1^{(l-1)\alpha_1+(2l-1)\alpha_{21}}X_4-\right.$ \newline $\left. X_2^{(2l-1)\alpha_2}X_3^{l-1}\right]=r_2$. $T_{r_2}=\{f_7\}$ and  
	   ${\rm spoly}(r_2,f_7)=X_1^{2\alpha_1+6\alpha_{21}+1}X_3^{\alpha_3-(l+1)}\left[ X_1^{(l-2)\alpha_1+(2l-3)\alpha_{21}}X_4-\right.$\newline$\left.X_2^{(2l-3)\alpha_2}X_3^{l-2}\right]=r_3$. $T_{r_3}=\{f_7\}$ and continuing inductively we obtain   $r_{l+1}={\rm spoly}(r_l,f_7)=X_1^{l\alpha_1+(2l+2)\alpha_{21}+1}X_3^{\alpha_3-(l+1)}f_2$ and hence $NF({\rm spoly}(g_0,h_l) \vert G)=0$ . 
	   
	    \item  $NF({\rm spoly}(h_j,h_l) \vert G)=0$ as the leading monomials are relatively prime.  
	     \item ${\rm spoly}(f_1,h_j)= X_1^{(j+1)\alpha_1+(2j+3)\alpha_{21}+1}X_3^{\alpha_3-(j+2)}-X_2^{(2j+2)\alpha_2+1}X_4^2=r_1$.$(6)$ and $j>s$  implies ${\rm LM}(r_1)=X_2^{(2j+2)\alpha_2+1}X_4^2$ and $T_{r_1}=\{f_5,g_s\}$. Since
	     ${\rm ecart}(f_5)$ is minimal by $(7)$,  ${\rm spoly}(r_1,f_5)=X_1^{\alpha_{21}+1}X_2^{(2j+2)\alpha_2}X_3^{\alpha_3-1}-X_1^{(j+1)\alpha_2+(2j+3)\alpha_{21}+1}X_3^{\alpha_3-(j+2)}=r_2$. $T_{r_2}=\{ f_7\}$. Then, ${\rm spoly}(r_2,f_7)=X_1^{\alpha_1+3\alpha_{21}+1}X_3^{\alpha_3-(j+2)}\left[ X_1^{j\alpha_1+2j\alpha_{21}}-X_2^{2j\alpha_2}X_3^{j} \right]=r_3$.    $T_{r_3}=\{ f_7\}$ and continuing inductively, we obtain $r_{j+2}={\rm spoly}(r_{j+1},f_7)=X_1^{(j+1)\alpha_1+(2j+1)\alpha_{21}+1}X_3^{\alpha_3-(j+2)} f_7$ and hence $NF({\rm spoly}(f_1,h_j) \vert G)=0$ .
	     \item  $NF({\rm spoly}(f_1,h_l) \vert G)=0$ as the leading monomials are relatively prime.   
	      \item ${\rm spoly}(f_2,h_j)=X_2^{\alpha_2}g_j$. Hence, by lemma \ref{gj}, $NF({\rm spoly}(f_2,h_j) \vert G)=0$.
	      \item $NF({\rm spoly}(f_2,h_l) \vert G)=0$ as the leading monomials are relatively prime.  
	     	\item ${\rm spoly}(f_3,h_j)=X_1^{(j+1)\alpha_1+(2j+2)\alpha_{21}}X_2-X_2^{(2j+2)\alpha_2+1}X_3^{j+1}$. Set this as $r_1$. Then ${\rm LM}(r_1)=X_2^{(2j+2)\alpha_2+1}X_3^{j+1}$ and $T_{r_1}=\left\lbrace f_7 \right\rbrace $. ${\rm spoly}(r_1,f_7)=X_1^{\alpha_1+2\alpha_{21}}X_2\left[ X_2^{2j\alpha_2}X_3^j-X_1^{(j)\alpha_1+2j\alpha_{21}}\right]=r_2 $. $T_{r_2}=\left\lbrace f_7 \right\rbrace $ and continuing inductively, we obtain $r_{j+1}={\rm spoly}(r_{j},f_7)=X_1^{j\alpha_1+2j\alpha_{21}}X_2f_7$ and hence $NF({\rm spoly}(f_3,h_j) \vert G)=0$ .
	     	  \item $NF({\rm spoly}(f_3,h_l) \vert G)=0$ as the leading monomials are relatively prime.  
	      \item$NF({\rm spoly}(f_4,h_j) \vert G)=0$ as the leading monomials are relatively prime.  
	     	  \item $NF({\rm spoly}(f_4,h_l) \vert G)=0$ as the leading monomials are relatively prime.
	     	    \item $NF({\rm spoly}(f_5,h_j) \vert G)=0$ as the leading monomials are relatively prime. 
	     	    \item ${\rm spoly}(f_5,h_l)=X_1^{l\alpha_1+(2l+3)\alpha_{21}+1}X_3^{\alpha_3-(l+1)}X_4^{2}-X_1^{\alpha_{21}+1}X_2^{(2l+2)\alpha_2}X_3^{\alpha_3-1}$. Set this as $r_1$. If  ${\rm LM}(r_1)=X_1^{l\alpha_1+(2l+3)\alpha_{21}+1}X_3^{\alpha_3-(l+1)}X_4^{2}$ then $T_{r_1}=\{f_1,f_2,f_7\}$ but since ${\rm ecart}(f_1)$ is minimal among these, ${\rm spoly}(r_1,f_1)=X_1^{\alpha_{21}+1}X_3^{\alpha_3-(l+2)} \left[ X_1^{(l+1)\alpha_1+(2l+2)\alpha_{21}}-X_2^{2(l+1)\alpha_2}X_3^{(l+1)} \right]=r_2 $ and $T_{r_2}=\{f_7\}$.  ${\rm spoly}(r_2,f_7)=X_1^{\alpha_{21}+1}X_2^{2\alpha_2}X_3^{\alpha_3-(l+1)} \left[ X_1^{(l)\alpha_1+(2l)\alpha_{21}}-X_2^{2(l)\alpha_2}X_3^{(l)} \right]=r_3 $ and $T_{r_3}=\{f_7\}$. Continuing inductively we obtain, ${\rm spoly}(r_{l+1},f_7)=X_1^{\alpha_{21}+1}X_2^{2l\alpha_2}X_3^{\alpha_3-2}f_7$ which implies  $NF({\rm spoly}(f_5,h_l) \vert G)=0$.  If  ${\rm LM}(r_1)=X_1^{\alpha_{21}+1}X_2^{(2l+2)\alpha_2}X_3^{\alpha_3-1}$ then $T_{r_1}=\{f_7\}$ and \newline
	     	    ${\rm spoly}(r_1,f_7)=X_1^{\alpha_1+3\alpha_{21}+1}X_3^{\alpha_3-(l+1)} \left[ X_2^{2l\alpha_2}X_3^{(l-1)}- X_1^{(l-1)\alpha_1+(2l)\alpha_{21}}X_4^2 \right]=r_2 $. Then, $T_{r_2}=\{f_7\}$ and  ${\rm spoly}(r_2,f_7)=X_1^{2\alpha_1+5\alpha_{21}+1}X_3^{\alpha_3-(l+1)} \left[ X_2^{2(l-1)\alpha_2}X_3^{(l-2)} -X_1^{(l-2)\alpha_1+(2l-2)\alpha_{21}}X_4^2\right]=r_3 $ and\newline $T_{r_3}=\{f_7\}$. Continuing inductively we obtain, $r_{l+1}={\rm spoly}(r_{l},f_7)=X_1^{l\alpha_1+(2l+1)\alpha_{21}+1}X_3^{\alpha_3-(l+1)}\left[\right.$  $\left. X_2^{2\alpha_2}-X_1^{2\alpha_{21}}X_4^2\right]$. $T_{r_{l+1}}=\{f_2\}$ and ${\rm spoly}(r_{l+1},f_2)=X_1^{l\alpha_1+(2l+1)\alpha_{21}+1}X_2^{\alpha_2}X_3^{\alpha_3-(l+1)}f_2$. Hence, $NF({\rm spoly}(f_5,h_l) \vert G)=0$ in this case, too.
	     	    
	     	   \item ${\rm spoly}(f_6,h_j)=g_{j+1}$ for all $s\leq j<l$. Hence  by lemma \ref{gj} $NF({\rm spoly}(f_6,h_j) \vert G)=0$.
	     	    \item ${\rm spoly}(f_6,h_l)=X_1^{\alpha_1+\alpha_{21}}X_2^{(2l+1)\alpha_2+1}-X_1^{l\alpha_1+(2l+3)\alpha_{21}+1}X_3^{\alpha_3-l}X_4=r_1$. ${\rm LM}(r_1)=X_1^{l\alpha_1+(2l+3)\alpha_{21}+1}X_3^{\alpha_3-l}X_4$  by the definition of $l$ and $(4).$ Then $T_{r_1}=\{ f_2\}$  and ${\rm spoly}(r_1,f_2)=X_1^{\alpha_1+\alpha_{21}}X_2^{\alpha_2}h_{j-1}$. Then $NF({\rm spoly}(f_6,h_l) \vert G)=0$
	     	   \item ${\rm spoly}(f_7,h_j)=h_{j+1}$ and hence $NF({\rm spoly}(f_7,h_j) \vert G)=0$
	     	    \item ${\rm spoly}(f_7,h_l)= X_1^{\alpha_1+2\alpha_{21}}h_{l-1}$ if $2\alpha_2+1<\alpha_1+2\alpha_{21}$. Otherwise leading monomials of $f_7$ and $h_l$ are relatively prime. As a result, in both cases, $NF({\rm spoly}(f_7,h_l) \vert G)=0$.
	\end{itemize}
	\begin{center}
		\underline{For $k=2$:}
	\end{center}
	In this case, since s might be greater than zero, $t$ will be greater than zero and $h_0$ will not be an element of the standard basis. Which means, from the above computations, only ${\rm spoly}(f_3,f_7)$ must be reconsidered. In addition to the normal forms considered in the case of $k=1$, we need the following for $k=2$ to prove the theorem:
	
\begin{itemize}
	\item  ${\rm spoly}(f_3,f_7)=X_1^{\alpha_1-\alpha_{21}-1}h_0$. The problem here is that, since $k=2,$ $s$ is not necessarily $0$ and we can not guarantee if $h_0 \in G$. Set ${\rm spoly}(f_3,f_7)$ as $r_1$. If $s>0$, then $l>0$ and by its definition, $3\alpha_{21}+\alpha_3<2\alpha_2+1$ and ${\rm LM}(r_1)=X_1^{\alpha_1+2\alpha_{21}}X_3^{\alpha_3-1}$ and $T_{r_1}=\{g_0\}$. ${\rm spoly}(r_1,g_0)=X_1^{\alpha_1-\alpha_{21}-1}X_2^{\alpha_2+1}f_2$. Hence, $NF({\rm spoly}(f_3,f_7) \vert G)=0$.
	\item ${\rm spoly}(g_i,g_j)= X_1^{(j-i)\alpha_1+2(j-i)\alpha_{21}}X_2^{(2i+1)\alpha_2+1}X_4-X_2^{(2j+1)\alpha_2+1}X_3^{j-i}X_4$ Set this as $r_1$. Then ${\rm LM}(r_1)= X_1^{(j-i)\alpha_1+2(j-i)\alpha_{21}}X_2^{(2i+1)\alpha_2+1}X_4$ and ${\rm spoly}(r_1,f_7)= X_1^{\alpha_1+2\alpha_{21}}X_2^{(2i+1)\alpha_2+1}X_4\left[ X_1^{(j-i-1)\alpha_1+2(j-i-1)\alpha_{21}}\right.$ $-\left.X_2^{2(j-i-1)\alpha_2}X_3^{j-i-1} \right]=r_2$ which implies that  $T_{r_2}=\{f_7\}$ and this, continuing inductively implies that  $r_{j-i}={\rm spoly}(r_{j-i-1},f_7)=$ $X_1^{(j-i-1)\alpha_1+2(j-i-1)\alpha_{21}}X_2^{[2i+1]\alpha_2+1}X_4f_7$  and hence $NF({\rm spoly}(g_i,g_j) \vert G)=0$ .
	
	\item   $NF({\rm spoly}(g_i,g_s) \vert G)=0$ as the leading monomials are relatively prime.   
	
	\item ${\rm spoly}(g_i,h_j)=X_1^{(j-i)\alpha_1+(2(j-i)+1)\alpha_{21}}X_2^{(2i+1)\alpha_2+1}X_4-X_2^{(2j+2)\alpha_2+1}X_3^{j-i}$. Set this as $r_1$. If  ${\rm LM}(r_1)=X_1^{(j-i)\alpha_1+(2(j-i)+1)\alpha_{21}}X_2^{(2i+1)\alpha_2+1}X_4$ then $T_{r_1}=\{ f_2\}$ and  ${\rm spoly}(r_1,f_2)=$\newline$X_2^{(2i+2)\alpha_2+1}\left[ X_1^{(j-i)\alpha_1+(2(j-i)+1)\alpha_{21}}X_4-X_2^{(2(j-i)+1)\alpha_2}X_3^{j-i}\right]=r_2 $.      ${\rm LM}(r_2)= X_2^{(2(j-i)+1)\alpha_2}X_3^{j-i} $ and $T_{r_2}=\{ f_7\}$.  ${\rm spoly}(r_2,f_7)=X_1^{\alpha_1+2\alpha_{21}}X_2^{2(i+1)\alpha_2+1}\left[ X_1^{(j-i-1)\alpha_1+(2(j-i-1))\alpha_{21}}-X_2^{2(j-i-1)\alpha_2}X_3^{j-i-1}\right]$\newline$=r_3 $. $T_{r_3}=\{f_7\}$ and continuing inductively, finally we obtain $r_{j-i+1}={\rm spoly}(r_{j-i},f_7)=$\newline$X_1^{(j-i-1)\alpha_{1}+2(j-i-1)\alpha_{21}}X_2^{2(i+1)\alpha_2+1}f_7$ and hence $NF({\rm spoly}(g_i,h_j) \vert G)=0$ in this case.\newline If ${\rm LM}(r_1)=X_2^{(2j+2)\alpha_2+1}X_3^{j-i}$, and  then $T_{r_1}=\{f_7\}$ and ${\rm spoly}(r_1,f_7)=$\newline $X_1^{\alpha_1+2\alpha_{21}}X_2^{(2i+1)\alpha_2+1}\left[ X_2^{(2(j-i-1)+1)\alpha_2}X_3^{j-i-1}-X_1^{(j-i-1)\alpha_{1}+(2(j-i-1)+1)\alpha_{21}}X_4 \right]=r_2 $. $T_{r_2}=\{f_7\}$ and ${\rm spoly}(r_2,f_7)= X_1^{2\alpha_1+4\alpha_{21}}X_2^{(2i+1)\alpha_2+1}\left[ X_2^{(2(j-i-2)+1)\alpha_2}X_3^{j-i-2}-X_1^{(j-i-2)\alpha_1+(2(j-i-2)+1)\alpha_{21}}X_4\right]=r_3 $. $T_{r_3}=\{f_7\}$ and continuing inductively, finally we obtain $r_{j-i+1}={\rm spoly}(f_7,r_{j-i})=$\newline$X_1^{(j-i)\alpha_1+2(j-i)\alpha_{21}}X_2^{(2i+1)\alpha_2+1}f_2$ and hence $NF({\rm spoly}(g_i,h_j) \vert G)=0$ in this case, too.

	\item  $NF({\rm spoly}(g_i,h_l) \vert G)=0$ as the leading monomials are relatively prime.    
	\item $NF({\rm spoly}(g_s,h_j) \vert G)=0$ as the leading monomials are relatively prime.    

	\item ${\rm spoly}(g_s,h_l)=X_1^{s\alpha_1+(2s+2)\alpha_{21}+1}X_2^{(2(l-s)+1)\alpha_2}X_3^{\alpha_3-(s+1)}-X_1^{l\alpha_1+(2l+3)\alpha_{21}+1}X_3^{\alpha_3-(l+1)}X_4=r_1$. If ${\rm LM}(r_1)=X_1^{s\alpha_1+(2s+2)\alpha_{21}+1}X_2^{(2(l-s)+1)\alpha_2}X_3^{\alpha_3-(s+1)}$, then $T_{r_1}=\{f_7\} $ and ${\rm spoly}(r_1,f_7)=$\newline $X_1^{(s+1)\alpha_1+(2s+4)\alpha_{21}+1}X_3^{\alpha_3-(l+1)}\left[ X_2^{(2(l-s-1)+1)\alpha_2}X_3^{l-s-1}-X_1^{(l-s-1)\alpha_1+(2(l-s-1)+1)\alpha_{21}}X_4\right]=r_2 $. $T_{r_2}=\{f_7\}$ and continuing inductively, $r_{l-s+1}={\rm spoly}(r_{l-s},f_7)=X_1^{l\alpha_1+2(l+1)\alpha_{21}+1}X_3^{\alpha_3-(l+1)}f_2$. 
If ${\rm LM}(r_1)=X_1^{l\alpha_1+(2l+3)\alpha_{21}+1}X_3^{\alpha_3-(l+1)}X_4$, then $T_{r_1}=\{f_2\}$ and \newline${\rm spoly}(r_1,f_2)=X_1^{s\alpha_1+(2s+2)\alpha_{21}+1}X_2^{\alpha_2}X_3^{\alpha_3-(l+1)}\left[ X_1^{(l-s)\alpha_1+2(l-s)\alpha_{21}}-X_2^{2(l-s)\alpha_2}X_3^{l-s}\right]=r_2 $. $T_{r_2}=\{f_7\}$ and continuing inductively, $T_{r_{l-s}}=\{f_7\}$ and $r_{l-s+1}={\rm spoly}(r_{l-s},f_7)=X_1^{(l-1)\alpha_1+2(l-1)\alpha_{21}+1}X_2^{\alpha_2}X_3^{\alpha_3-(l+1)}f_7$. Hence, in both cases, $NF({\rm spoly}(g_s,h_l) \vert G)=0$.

\item ${\rm spoly}(f_1,g_i)= X_2^{(2i+1)\alpha_2+1}X_4^{3}-X_1^{(i+1)\alpha_1+(2i+2)\alpha_{21}+1}X_3^{\alpha_3-(i+2)}=r_1$. ${\rm LM}(r_1)=X_2^{(2i+1)\alpha_2+1}X_4^{3}$ by $(7)$ and $(3)$ and $T_{r_1}=\{f_4,f_5\}$ but ${\rm ecart}(f_5)$ is minimal. Then ${\rm spoly}(r_1,f_5)=X_1^{\alpha_{21}+1}X_2^{(2i+1)\alpha_2}X_3^{\alpha_3-1}X_4^{}-X_1^{(i+1)\alpha_1+(2i+2)\alpha_{21}+1}X_3^{\alpha_3-(i+2)}=r_2$. ${\rm LM}(r_2)=X_1^{\alpha_{21}+1}X_2^{(2i+1)\alpha_2}X_3^{\alpha_3-1}X_4^{}$ and $T_{r_2}=\{f_2,f_6,f_7\}$ and ${\rm spoly}(r_2,f_2)=X_1X_3^{\alpha_{3}-(i+2)}\left[ X_2^{(2i+2)\alpha_2}X_3^{i+1}-X_1^{(i+1)\alpha_1+2(i+1)\alpha_{21}}\right] =r_3$. ${\rm LM }(r_3)=X_1X_2^{(2i+2)\alpha_2}X_3^{\alpha_3-1} $ and $T_{r_3}=\{f_7\}$, ${\rm spoly}(r_3,f_7)=X_1^{\alpha_1+2\alpha_{21}+1}X_3^{\alpha_{3}-(i+2)}\left[ X_2^{(2i)\alpha_2}X_3^{i}-X_1^{i\alpha_1+2i\alpha_{21}}\right]=r_4$. $T_{r_4}=\{f_7\}$ and continuing inductively, we obtain $r_{i+1}={\rm spoly}(r_{i+2},f_7)=X_1^{i\alpha_1+2i\alpha_{21}+1}X_3^{\alpha_3-(i+2)}f_7$ and hence, $NF({\rm spoly}(f_1,g_i) \vert G)=0$

\item ${\rm spoly}(f_2,g_i)= X_2^{(2i+1)\alpha_2+1}X_4^{2}-X_1^{i\alpha_1+(2i+1)\alpha_{21}+1}X_2^{\alpha_2}X_3^{\alpha_3-(i+1)} =r_1$. ${\rm LM}(r_1)=X_2^{(2i+1)\alpha_2+1}X_4^{2}$ by $(5)$ and $(7)$. $T_{r_1}=\{f_5\}$. ${\rm spoly}(r_1,f_5)=X_1^{\alpha_{21}+1}X_2^{\alpha_2}\left[ X_2^{2i\alpha_2}X_3^{i}-X_1^{i\alpha_1+2i\alpha_{21}}\right]=r_2 $. $T_{r_2}=\{f_7\}$.  ${\rm spoly}(r_2,f_7)=X_1^{\alpha_1+3\alpha_{21}+1}X_2^{\alpha_2}X_3^{\alpha_3-(i+1)}\left[ X_2^{2(i-1)\alpha_2}X_3^{i-1}-X_1^{(i-1)\alpha_1+2(i-1)\alpha_{21}}\right]=r_3 $. $T_{r_3}=\{f_7\}$ and continuing inductively, $r_{i+1}={\rm spoly}(r_i,f_7)=X_1^{(i-1)\alpha_1+(2i+1)\alpha_{21}+1}X_2^{\alpha_2}X_3^{\alpha_3-(i+1)}f_7$ and hence $NF({\rm spoly}(f_2,g_i) \vert G)=0$

\item ${\rm spoly}(f_3,g_i)=X_2^{(2i+1)\alpha_2+1}X_3^{i+1}X_4-X_1^{(i+1)\alpha_1+(2i+1)\alpha_{21}}X_2=r_1$.  Using $(7)$ and $(4)$, ${\rm LM}(r_1)=X_2^{(2i+1)\alpha_2+1}X_3^{i+1}X_4$ and $T_{r_1}=\{f_6,f_7\}$ but ${\rm ecart}(f_7)$ is minimal. Then \newline${\rm spoly}(r_1,f_7)=X_1^{\alpha_1+2\alpha_{21}}X_2\left[ X_2^{(2i-1)\alpha_2}X_3^{i}X_4^{}-X_1^{i\alpha_1+(2i-1)\alpha_{21}}\right]=r_2 $.  $T_{r_2}=\{f_6,f_7\}$. Continuing inductively, we obtain, $r_{i+1}={\rm spoly}(r_i,f_7)=X_1^{i\alpha_1+2i\alpha_{21}}X_2^{}f_6$ and hence $NF({\rm spoly}(f_3,g_i) \vert G)=0$


\item  $NF({\rm spoly}(f_4,g_i) \vert G)=0$ as the leading monomials are relatively prime.    
\item  $NF({\rm spoly}(f_5,g_i) \vert G)=0$ as the leading monomials are relatively prime. 
\item ${\rm spoly}(f_6,g_i)=X_2^{(2i+2)\alpha_2+1}X_4^{2}-X_1^{(i+1)\alpha_1+(2i+3)\alpha_{21}+1}X_3^{\alpha_3-(i+2)}=r_1$. ${\rm LM}(r_1)=X_2^{(2i+2)\alpha_2+1}X_4^{2}$ by $(5)$ and $(7)$. Then $T_{r_1}=\{f_5\}$ and ${\rm spoly}(r_1,f_5)=X_1^{\alpha_{21}+1}X_2^{(2i+2)\alpha_2}X_3^{\alpha_3-1}-X_1^{(i+1)\alpha_1+(2i+3)\alpha_{21}+1}X_3^{\alpha_3-(i+2)}=r_2$. ${\rm LM}(r_2)=X_1^{\alpha_{21}+1}X_2^{(2i+2)\alpha_2}X_3^{\alpha_3-1}$ and $T_{r_2}=\{f_7\}$.\newline ${\rm spoly}(r_2,f_7)=X_1^{\alpha_1+3\alpha_{21}+1}X_3^{\alpha_3-(i+2)}\left[ X_1^{i\alpha_1+2i\alpha_{21}}-X_2^{2i\alpha_2}X_3^{i}\right]=r_3 .$ $T_{r_3}=\{f_7\}$ and continuing inductively we obtain, $r_{i+2}={\rm spoly}(r_{i+1},f_7)=X_1^{i\alpha_1+(2i+1)\alpha_{21}+1}X_3^{\alpha_3-(i+2)}f_7$ and hence  $NF({\rm spoly}(f_6,g_i) \vert G)=0$.
\item ${\rm spoly}(f_7,g_i)=g_{i+1}$. Hence, $NF({\rm spoly}(f_7,g_i) \vert G)=0$   

\item ${\rm spoly}(f_1,g_s)=X_1^{\alpha_1}X_2^{(2s+1)\alpha_2+1}-X_1^{s\alpha_1+(2s+2)\alpha_{21}+1}X_3^{\alpha_3-s}X_4$. Set this as $r_1$. Since $s-1<l$, by the definition of $l$ and $(4)$, ${\rm LM}(r_1)= X_1^{s\alpha_1+(2s+2)\alpha_{21}+1}X_3^{\alpha_3-s}X_4$. $T_{r_1}=\{f_2\}$ and ${\rm spoly}(r_1,f_2)=X_1^{s\alpha_1+(2s+1)\alpha_{21}+1}X_2^{\alpha_2}X_3^{\alpha_3-s}-X_1^{\alpha_1}X_2^{(2s+1)\alpha_2+1}=r_2$. $T_{r_2}=\{ g_{s-1}\}$  and ${\rm spoly}(r_2,g_{s-1})=X_1^{\alpha_1}X_2^{2s\alpha_2+1}f_2$. Hence, $NF({\rm spoly}(f_1,g_s) \vert G)=0$
\item ${\rm spoly}(f_2,g_s)=h_s$. Since $h_s \in G,$  $NF({\rm spoly}(f_2,g_s) \vert G)=0$
\item $NF({\rm spoly}(f_3,g_s) \vert G)=0$ as the leading monomials are relatively prime. 
\item ${\rm spoly}(f_4,g_s)=  X_1X_2^{(2s+2)\alpha_2}X_3^{\alpha_3-1}-X_1^{s\alpha_1+(2s+2)\alpha_{21}+1}X_3^{\alpha_3-(s+1)}X_4^2$. Set this as $r_1$. If ${\rm LM}(r_1)=X_1X_2^{(2s+2)\alpha_2}X_3^{\alpha_3-1}$, then $T_{r_1}=\{f_7\}$ and\newline
${\rm spoly}(r_1,f_7)=X_1^{\alpha_1+2\alpha_{21}+1}X_3^{\alpha_3-(s+1)}$ $\left[ X_2^{(2s)\alpha_2}X_3^{s-1}-X_1^{(s-1)\alpha_1+(2s)\alpha_{21}}X_4^2 \right]=r_2 $. $T_{r_2}=\{f_2\}$ and ${\rm spoly}(r_2,f_2)=X_1^{}X_2^{}X_3^{}\left[ X_2^{(2s-1)\alpha_2}X_3^{s-1}-X_1^{(s-1)\alpha_1+(2s-1)\alpha_{21}}X_4\right] =r_3$. $T_{r_3}=\{f_2\}$ and ${\rm spoly}(r_3,f_2)=X_1^{\alpha_1+2\alpha_{21}+1}X_2^{2\alpha_2}X_3^{\alpha_3-(s+1)}\left[X_1^{(s-1)\alpha_1+(2s-2)\alpha_{21}}-X_2^{(2s-2)\alpha_2}X_3^{s-1} \right]=r_4 $. $T_{r_4}=\{f_7\}$ and continuing inductively, $r_{s+2}={\rm spoly}(r_s,f_7)=X_1^{s\alpha_1+2\alpha_{21}+1}X_2^{(s)\alpha_2}X_3^{\alpha_3-(s+1)}f_7$.
If ${\rm LM}(r_1)=X_1^{s\alpha_1+(2s+2)\alpha_{21}+1}X_3^{\alpha_3-(s+1)}X_4^2$, then $T_{r_1}=\{ f_1,f_2 \}$ but ${\rm ecart }(f_2)$ is minimal. \newline ${\rm spoly}(r_1,f_2)=X_1X_2^{\alpha_2}X_3^{\alpha_3-(s+1)}\left[ X_1^{s\alpha_1+(2s+1)\alpha_{21}}X_4-X_2^{(2s+1)\alpha_2}X_3^s \right] $. Set this as $r_2$. Observe that $X_2^{\alpha_2}{\rm spoly}(f_5,g_s)-X_1^{\alpha_{21}}r_2=0$ and $NF({\rm spoly}(f_5,g_s) \vert G)=0$ (see below).\newline
Hence, in both of the cases, $NF({\rm spoly}(f_4,g_s) \vert G)=0$

\item ${\rm spoly}(f_5,g_s)=X_1^{\alpha_{21}+1}X_2^{(2s+1)\alpha_2}X_3^{\alpha_3-1}-X_1^{s\alpha_1+(2s+2)\alpha_{21}+1}X_3^{\alpha_3-(s+1)}X_4$. Set this as $r_1$.
If ${\rm LM}(r_1)=X_1^{\alpha_{21}+1}X_2^{(2s+1)\alpha_2}X_3^{\alpha_3-1}$, then $T_{r_1}=\{f_7\}$ and\newline
${\rm spoly}(r_1,f_7)=X_1^{\alpha_1+3\alpha_{21}+1}X_3^{\alpha_3-(s+1)}$ $\left[ X_2^{(2s-1)\alpha_2}X_3^{s-1}-X_1^{(s-1)\alpha_1+(2s-1)\alpha_{21}}X_4 \right]=r_2 $. $T_{r_2}=\{f_2\}$ and ${\rm spoly}(r_2,f_2)=X_1^{\alpha_1+3\alpha_{21}+1}X_2^{\alpha_2}X_3^{\alpha_3-(s+1)}\left[X_1^{(s-1)\alpha_1+(2s-2)\alpha_{21}}-X_2^{(2s-2)\alpha_2}X_3^{s-1} \right]=r_3 $. $T_{r_3}=\{f_7\}$ and continuing inductively, $r_{s+1}={\rm spoly}(r_s,f_7)=X_1^{s\alpha_1+(2s-1)\alpha_{21}+1}X_2^{\alpha_2}X_3^{\alpha_3-(s+1)}f_7$.
If ${\rm LM}(r_1)=X_1^{s\alpha_1+(2s+2)\alpha_{21}+1}X_3^{\alpha_3-(s+1)}X_4$, then $T_{r_1}=\{ f_2\}$
${\rm spoly}(r_1,f_2)=X_1^{\alpha_{21}+1}X_2^{\alpha_2}X_3^{\alpha_3-(s+1)}$ $\left[ X_1^{s\alpha_1+2s\alpha_{21}}-X_2^{2s\alpha_2}X_3^{s} \right]=r_2 $. $T_{r_2}=\{f_7\}$ and\newline ${\rm spoly}(r_2,f_7)=X_1^{\alpha_1+3\alpha_{21}+1}X_2^{\alpha_2}X_3^{\alpha_3-(s+1)}\left[X_1^{(s-1)\alpha_1+(2s-2)\alpha_{21}}-X_2^{(2s-2)\alpha_2}X_3^{s-1} \right]=r_3 $. $T_{r_3}=\{f_7\}$ and continuing inductively, $r_{s+1}={\rm spoly}(r_s,f_7)=X_1^{s\alpha_1+(2s-1)\alpha_{21}+1}X_2^{\alpha_2}X_3^{\alpha_3-(s+1)}f_7$ in this case, too. Hence, in both of the cases $NF({\rm spoly}(f_5,g_s) \vert G)=0$

\item ${\rm spoly}(f_6,g_s)= X_1^{\alpha_1+\alpha_{21}}h_s$ and hence $NF({\rm spoly}(f_6,g_s) \vert G)=0$
\item ${\rm spoly}(f_7,g_s)=X_1^{\alpha_1+2\alpha_{21}} g_{s-1}$ and hence $NF({\rm spoly}(f_7,g_s) \vert G)=0$    
\end{itemize}	

	Since all normal forms reduce to zero, $G$ is a standard basis  for $I_C$ 
\end{proof}
\begin{corollary}
	$\{{f_1}^*,{f_2}^*,...,{f_7}^*, {g_0}^*,...,{g_{s}}^*, {h_s}^*,...,{h_{l}}^*\}$ is a standard basis  for $I_C^*$  where  ${f_1}^*=X_3X_4^2$, ${f_2}^*=X_1^{\alpha_{21}}X_4$,  ${f_3}^*=X_3^{\alpha_3}$, ${f_4}^*=X_4^{3}$, ${f_5}^*=X_2X_4^2$, ${f_6}^*=X_2^{\alpha_2}X_3X_4$, $f_{7}^*=X_2^{2\alpha_2}X_3$ and ${g_{i}}^*=X_1^{i\alpha_1+(2i+2)\alpha_{21}+1}X_3^{\alpha_3-(i+1)}$ for $i=1,2,...,s-1$, ${g_{s}}^*=X_2^{(2s+1)\alpha_{2}+1}X_4$, ${h_{j}}^*=X_1^{j\alpha_1+(2j+3)\alpha_{21}+1}X_3^{\alpha_3-(j+1)}$ for $j=s,s+1,...,l-1$, ${h_{l}}^*= X_2^{(2l+2)\alpha_{2}+1}$ . Since $X_1 | {f_2}^*$, the tangent cone is not Cohen-Macaulay by the criterion given in \cite{AMS} .
\end{corollary}




\section{Hilbert Function}\label{3}

	Let $P(I_S^*)$ denote the numerator of the Hilbert series of $A/{I^*_S}$
\begin{theorem}\label{hilser}
	The numerator of the Hilbert series of the local ring $R_S$ is 
	
	$P(I_S^*)=1-3t^3+3t^4-t^5-t^{\alpha_{21}+1}(1-t)(1+t-2t^2+t^3)-t^{\alpha_3}(1-t^{\alpha_{21}+1}-t^2(1-t^{\alpha_{21}}))-t^{\alpha_{2}+2}(1-t)(1-t^{\alpha_{21}})(1-t^{\alpha_3-1})-t^{2\alpha_2+1}(1-t)(1-t^{\alpha_3-1})-(1-t)^2(1-t^{2\alpha_2})R(t)-(1-t)^2(1-t^{\alpha_{21}})t^{(2s+1)\alpha_2+2}-(1-t)^2t^{(2l+2)\alpha_2+1}$ where  $R(t)=t^{2\alpha_{21}+\alpha_3}\sum_{j=0}^{s-1} t^{j(\alpha_1+2\alpha_{21}-1)}+t^{s\alpha_1+(2s+3)\alpha_{21}+\alpha_3-s}\sum_{j=0}^{l-s-1} t^{j(\alpha_1+2\alpha_{21}-1)}$ for any $s>0$ and $R(t)=0$ if $s=0$ and $l=0$ .
\end{theorem}
\begin{proof}
	To compute the Hilbert series, we use Algorithm 2.6 of \cite{bayer} that is formed by continuous use of the proposition 
	
	"If $I$ is a monomial ideal with $I=<J,w>$, then the numerator of the Hilbert series of $A/I$ is $P(I)=P(J)-t^{\deg w}P(J:w)$ and $P(w)=1-t^{\deg w}$ where $w$ is a monomial and $\deg w$ is the total degree of $w$."

Taking $w_1={h_l}^*$, $w_2={h_s}^*$, $w_3={h_{l-1}}^*$,  $w_4={h_{l-2}}^*, \cdots, w_{l-s+2}={h_{s}}^*, w_{l-s+3}={g_{s}}^*$,..., $w_{l+2}={g_0}^*$, $w_{l+3}={f_7}^*$, $w_{l+4}={f_6}^*$, $w_{l+5}={f_3}^*$, $w_{l+6}={f_2}^*$, $w_{l+7}={f_4}^*$, $w_{l+8}={f_5}^*$, $w_{l+9}={f_1}^*$, If we set $J_0=I_S^*$, $J_{i+1}=J_{i}-\{ w_{i+1}\}$ for $i=0,\cdots,t+8$ in the Algorithm, we get  $P(J_i)=P(J_{i+1})-t^{\deg w_{i+1}}P(J_i:w_{i+1})$ and we obtain the desired result.
\end{proof}

\begin{corollary}
The second Hilbert series of the local ring is $Q(t)=(1+t+t^2+\cdots +t^{\alpha_{21}-1})(t+2t^2+t^4+t^5+\cdots+t^{\alpha_3})+(1+t+t^2+\cdots+t^{\alpha_3-2})(1+t+t^2+\cdots+t^{2\alpha_2}-t^{\alpha_2+2}(1+t+\cdots+t^{\alpha_{21}-1}))-(1+t+\cdots+t^{\alpha_{21}-1})t^{(2s+1)\alpha_2+2}+t^{\alpha_3-1}(1+t+\cdots+t^{(2l+2)\alpha_2-\alpha_3+1})-(1+t+\cdots+t^{2\alpha_{2}-1})R(t)$
\end{corollary}
\begin{proof}
Since $P(I_S^*)=\displaystyle\frac{Q(t)}{(1-t)^3}$, the result is a direct consequence of theorem \ref{hilser}.
\end{proof}
Clearly, since the krull dimension is one, if there are no negative terms in the second Hilbert Series, then the Hilbert function will be non-decreasing. We can state and prove the next theorem.
\begin{theorem}
	The local ring $R_S$ has a non-decreasing Hilbert function if $l=0$.
\end{theorem}
\begin{proof}
Since $\alpha_2>\alpha_{21}+1$ by $(4)$, $2\alpha_2>\alpha_2+\alpha_{21}+1$ and hence, 

\noindent $Q(t)=(1+t+t^2+\cdots +t^{\alpha_{21}-1})(t+2t^2+t^4+t^5+\cdots+t^{\alpha_3})+(1+t+t^2+\cdots+t^{\alpha_3-2})(1+t+t^2+\cdots+t^{\alpha_2+1}+t^{\alpha_2+\alpha_{21}+2}+\cdots+t^{2\alpha_{2}})-(1+t+\cdots+t^{\alpha_{21}-1})t^{(2s+1)\alpha_2+2}+t^{\alpha_3-1}(1+t+\cdots+t^{(2l+2)\alpha_2-\alpha_3+1})-(1+t+\cdots+t^{2\alpha_{2}-1})R(t)$

\noindent When $l=0$, then $s=0$ and $R(t)=0$. There are two cases:

\noindent If $\alpha_3\geq \alpha_2+3$:

\noindent $Q(t)=(1+t+t^2+\cdots+t^{\alpha_{21}-1})\left[ t+2t^2+t^4+\cdots+t^{\alpha_2+1}+t^{\alpha_2+3}+\cdots+t^{\alpha_3}\right] +(1+t+\cdots+t^{\alpha_3-2})\left[ 1+t+\cdots+t^{\alpha_2+1} \right. $\newline $\left.+t^{\alpha_2+\alpha_{21}+2}+\cdots+t^{2\alpha_2}\right]+t^{\alpha_3-1}\left[ 1+t+\cdots+t^{2\alpha_2-\alpha_3+1}\right]$

\noindent If $\alpha_3 < \alpha_2+3$:

\noindent $Q(t)=(1+t+t^2+\cdots+t^{\alpha_{21}-1})\left[ t+2t^2+t^4+\cdots+t^{\alpha_3}\right] +(1+t+\cdots+t^{\alpha_3-2})\left[ 1+t+\cdots+t^{\alpha_2+1}+t^{\alpha_2+\alpha_{21}+2}+\cdots+ \right. $\newline $\left.t^{2\alpha_2}\right]+\left[t^{\alpha_3-1}+\cdots+t^{\alpha_2+1}+t^{\alpha_2+\alpha_{21}+2}+\cdots+t^{2\alpha_2}\right]$

In both cases, there are no negative terms in the second Hilbert series hence the Hilbert Function is nondecreasing.

\noindent When $l>0$, since $l \geq s$, and $\alpha_2>\alpha_{21}+1$, $(2l+2)\alpha_2>(2s+1)\alpha_2+\alpha_{21}+1$ and $\alpha_3<(2s+1)\alpha_2+3$ which means that all of the negative terms $-(1+t+\cdots+t^{\alpha_{21}-1})t^{(2s+1)\alpha_2+2}$ in $Q(t)$ will be cancelled out by the terms in $t^{\alpha_3-1}(1+t+\cdots+t^{(2l+2)\alpha_2-\alpha_3+1})$.

\noindent Then it is enough to show that the negative terms $-(1+t+\cdots+t^{2\alpha_{2}-1})R(t)$ will also be cancelled out. Note that,

$(1+t+\cdots+t^{2\alpha_{2}-1})R(t)=(1+t+\cdots+t^{2\alpha_{2}-1})(t^{2\alpha_{21}+3}+\text{ higher degree terms }+ t^{(l-1)\alpha_1+(2l+1)\alpha_{21}+\alpha_3-l+1}$$=t^{2\alpha_{21}+3}+\text{ some higher degree terms}+t^{(l-1)\alpha_1+(2l+1)\alpha_{21}+\alpha_3-l+2\alpha_2}.$
	
\noindent Since $l$ is the smallest integer with $l\alpha_1+(2l+3)\alpha_{21}+\alpha_3\geq (2l+2)\alpha_2+l+1$, for $l-1$ we have, $(l-1)\alpha_1+(2l+1)\alpha_{21}+\alpha_3<2l\alpha_{21}+l\implies (l-1)\alpha_1+(2l+1)\alpha_{21}+\alpha_3-l+2\alpha_2<(2l+2)\alpha_2$.

Also, since $2\alpha_2-1<\alpha_1+2\alpha_{21}-1$, all of the terms in $(1+t+\cdots+t^{2\alpha_{2}-1})R(t)$ has coefficient 1.

Hence, all of the negative terms disappear in $Q(t)$ and the Hilbert Function is nondecreasing.
\end{proof}

\section{Examples}\label{4}

\begin{example}Let $\alpha_{21}=12 , \alpha_1=38 , \alpha_2=20 ,\alpha_3=8, \alpha_{4}=3 $. Then $k=1$, $s=0$ and $l=0$ and the corresponding standard basis is $\{ f_1=X_1^{38}-X_3X_4^2,
	f_2=X_2^{20}-X_1^{12}X_4,
	f_3=X_3^8-X_1^{25}X_2,
	f_4=X_4^3-X_1X_2^{19}X_3^7,
	f_5=X_1^{13}X_3^{7}-X_2X_4^{2},
	f_6=X_1^{50}-X_2^{20}X_3X_4,
	f_7=X_1^{62}-X_2^{40}X_3,
	g_0=X_2^{21}X_4-X_1^{25}X_3^{7}, 
	h_0=X_2^{41}-X_1^{37}X_3^{7} \}$. The fisrt Hilbert series is $P(I_S^*)=1-3t^3+3t^4-t^5-t^8+t^{10}-t^{13}+3t^{15}-3t^{16}+t^{17}+t^{21}-3t^{22}+3t^{23}-t^{24}+t^{29}-t^{30}+2t^{34}-3t^{35}+t^{36}-3t^{41}+4t^{42}-t^{43}+t^{48}-t^{49}$ and the second Hilbert series is
	$Q(t)=1+3t^{}+6t^{2}+7t^{3}+9t^{4}+11t^{5}+13t^{6}+15t^{7}+16t^{8}+16t^{9}+16t^{10}+16t^{11}+16t^{12}+15t^{13}+13t^{14}+13t^{15}+12t^{16}+11t^{17}+10t^{18}+9t^{19}+8t^{20}+8t^{21}+6t^{22}+5t^{23}+4t^{24}+3t^{25}+2t^{26}+t^{27}+2t^{34}+3t^{35}+4t^{36}+5t^{37}+6t^{38}+7t^{39}+8t^{40}+6t^{41}+5t^{42}+4t^{43}+3t^{44}+2t^{45}+t^{46}$. Since there are no negative terms, Hilbert function is nondecreasing.
\end{example}

\begin{example}For $\alpha_{21}=11$, $\alpha_{1}=62$,$\alpha_{2}=40$, $\alpha_{3}=14$, $\alpha_{4}=3$, we have $k=2$, $l=12$ and $s=3$. Corresponding standard basis is:

 $\{ f_1=X_1^{62}-X_3X_4^2,
	f_2=X_2^{40}-X_1^{11}X_4,
	f_3=X_3^{14}-X_1^{50}X_2,
	f_4=X_4^3-X_1X_2^{39}X_3^{13},
	f_5=X_1^{12}X_3^{13}-X_2X_4^{2},
	f_6=X_1^{73}-X_2^{40}X_3X_4,
	f_7=X_1^{84}-X_2^{80}X_3,
	g_0=X_2^{41}X_4-X_1^{23}X_3^{13}, 
	g_1=X_2^{121}X_4-X_1^{107}X_3^{12},
	g_2=X_2^{201}X_4-X_1^{191}X_3^{11},
	g_3=X_2^{281}X_4-X_1^{275}X_3^{10},
	h_4=X_2^{321}-X_1^{286}X_3^{10},
	h_5=X_2^{401}-X_1^{370}X_3^{9},
	h_6=X_2^{481}-X_1^{454}X_3^{8},
	h_7=X_2^{561}-X_1^{538}X_3^{7},
	h_8=X_2^{641}-X_1^{622}X_3^{6},
	h_9=X_2^{721}-X_1^{706}X_3^{5},
	h_{10}=X_2^{801}-X_1^{790}X_3^{4},
	h_{11}=X_2^{881}-X_1^{874}X_3^{3},
	h_{12}=X_2^{961}-X_1^{958}X_3^{2},
	h_{13}=X_2^{1041}-X_1^{1042}X_3, \}$. 
The first Hilbert series is $P(I_S^*)=1-3t^3+3t^4-t^5-t^{12}+2t^{14}-3t^{15}+2t^{16}+t^{26}-t^{27}-t^{36}+2t^{37}-t^{38}-t^{42}+t^{43}+t^{53}-t^{54}+t^{55}-t^{56}-t^{66}+t^{67}-t^{81}+t^{82}+t^{94}-t^{95}+t^{116}-2t^{117}+t^{118}-t^{119}+2t^{120}-t^{121}+t^{199}-2t^{200}+t^{201}-t^{202}+2t^{203}-t^{204}+2t^{293}-2t^{294}+t^{295}-t^{296}+2t^{297}-t^{298}+t^{376}-2t^{377}+t^{378}-t^{379}+2t^{380}-t^{381}+t^{459}-2t^{460}+t^{461}-t^{462}+2t^{463}-t^{464}+t^{542}-2t^{543}+t^{544}-t^{545}+2t^{546}-t^{547}+t^{625}-2t^{626}+t^{627}-t^{628}+2t^{629}-t^{630}+t^{708}-2t^{709}+t^{710}-t^{711}+2t^{712}-t^{713}+t^{791}-t^{792}+t^{793}-t^{794}+2t^{795}-t^{796}+t^{824}-2t^{875}+t^{876}-t^{877}+2t^{878}-t^{879}+t^{957}-2t^{958}+t^{959}-t^{960}+2t^{961}-t^{962}+t^{1040}-3t^{1041}+3t^{1042}-t^{1043}$ and the second Hilbert series is
	$Q(t)=1+3t^{}+6t^{2}+7t^{3}+9t^{4}+11t^{5}+13t^{6}+15t^{7}+17t^{8}+19t^{9}+21t^{10}+23t^{11}+24t^{12}+24t^{13}+25t^{14}+24t^{15}+23t^{16}+22t^{17}+21t^{18}+20t^{19}+19t^{20}+18t^{21}+17t^{22}+16t^{23}+15t^{24}+14t^{25}+14t^{26}+14t^{27}+14t^{28}+14t^{29}+14t^{30}+14t^{31}+14t^{32}+14t^{33}+14t^{34}+14t^{35}+13t^{36}+13t^{37}+13t^{38}+13t^{39}+13t^{40}+13t^{41}+12t^{42}+11t^{43}+10t^{44}+9t^{45}+8t^{46}+7t^{47}+6t^{48}+5t^{49}+4t^{50}+3t^{51}+2t^{52}+2t^{53}+2t^{54}+3t^{55}+4t^{56}+5t^{57}+6t^{58}+7t^{59}+8t^{60}+9t^{61}+10t^{62}+11t^{63}+12t^{64}+13t^{65}+13t^{66}+13t^{67}+13t^{68}+13t^{69}+13t^{70}+13t^{71}+13t^{72}+13t^{73}+13t^{74}+13t^{75}+13t^{76}+13t^{77}+13t^{78}+13t^{79}+13t^{80}+12t^{81}+11t^{82}+10t^{83}+9t^{84}+8t^{85}+7t^{86}+6t^{87}+5t^{88}+4t^{89}+3t^{90}+2t^{91}+t^{92}+t^{116}+t^{117}+t^{118}+t^{199}+t^{200}+t^{201}+t^{293}+t^{294}+t^{295}+t^{376}+t^{377}+t^{378}+t^{459}+t^{460}+t^{461}+t^{542}+t^{543}+t^{544}+t^{625}+t^{626}+t^{627}+t^{708}+t^{709}+t^{710}+t^{791}+t^{792}+t^{793}+t^{874}+t^{875}+t^{876}+t^{957}+t^{958}+t^{959}+t^{1040}$. Since there are no negative terms, Hilbert function is nondecreasing.

\end{example}

\end{document}